\documentclass[10pt, reqno]{amsart}
\usepackage{amsmath, amsthm, amscd, amsfonts, amssymb, graphicx, color}
\usepackage[bookmarksnumbered, colorlinks, plainpages]{hyperref}
\hypersetup{colorlinks=true,linkcolor=red, anchorcolor=green, citecolor=cyan, urlcolor=red, filecolor=magenta, pdftoolbar=true}
\usepackage{mathrsfs}

\newtheorem{theorem}{Theorem}[section]
\newtheorem{lemma}[theorem]{Lemma}

\newtheorem{corollary}[theorem]{Corollary}
\theoremstyle{definition}
\newtheorem{definition}[theorem]{Definition}

\theoremstyle{remark}
\newtheorem{remark}[theorem]{Remark}
\numberwithin{equation}{section}

\begin{document}
\setcounter{page}{1}

\title[K-theory of Etesi $C^*$-algebras]{K-theory of Etesi $C^*$-algebras}

\author[Igor V. Nikolaev]
{Igor V. Nikolaev$^1$}

\address{$^{1}$ Department of Mathematics and Computer Science, St.~John's University, 8000 Utopia Parkway,  
New York,  NY 11439, United States.}
\email{\textcolor[rgb]{0.00,0.00,0.84}{igor.v.nikolaev@gmail.com}}

\dedicatory{All data are available as part of the manuscript}




\begin{abstract}
We study the $C^*$-algebra $\mathbb{E}_{\mathscr{M}}$ of a smooth 
 4-dimensional manifold $\mathscr{M}$ introduced by G\'abor Etesi.
  It is  proved  that the  $\mathbb{E}_{\mathscr{M}}$ is a stationary  AF-algebra.  
 We calculate the topological and smooth invariants of  $\mathscr{M}$  
 in terms of the K-theory  of the  $C^*$-algebra $\mathbb{E}_{\mathscr{M}}$.
  Using Gompf's Stable Diffeomorphism Theorem,   it is shown that 
   all smoothings of  $\mathscr{M}$
  form a torsion abelian group.  The latter is
  isomorphic  to the Brauer group of a number field 
 associated to the K-theory of  $\mathbb{E}_{\mathscr{M}}$. 
 \end{abstract}

\maketitle

\section{Introduction}
Let $\mathscr{M}$ be a topological  4-dimensional manifold.
Unlike dimensions 2 and 3, the smooth structures on  $\mathscr{M}$  are
detached from the topology of  $\mathscr{M}$. 
Due to the works of Rokhlin, Freedman and Donaldson, it is known that
 $\mathscr{M}$ can be  non-smooth and    if there exists a smooth structure, it needs not be unique. 
 Further we assume $\mathscr{M}$ to be compact and smoothable. 
The classification of all smoothings of $\mathscr{M}$  is an open problem. 
Denote by $\mathit{Diff}(\mathscr{M})$ 
a group of the orientation-preserving diffeomorphisms of $\mathscr{M}$
and let $\mathit{Diff}_0(\mathscr{M})$ be a connected component of $\mathit{Diff}(\mathscr{M})$
containing the identity.  The group  $G:=\mathit{Diff}(\mathscr{M})/ \mathit{Diff}_0(\mathscr{M})$
is discrete and therefore locally compact. 
\begin{definition}\label{dfn1.1}
The Etesi $C^*$-algebra  
is  a group $C^*$-algebra $C^*(G)$ of the locally compact amenable group 
$G$;   see Remarks \ref{rmk1.2} and \ref{rmk4.1}.
\end{definition}
\begin{remark}\label{rmk1.2}
$G$ is a countable, discrete, amenable group acting on $\mathscr{M}$ and 
the action  admits a faithful $G$-invariant Borel probability measure, e.g. by taking
the Lebesgue measure of the orbit space of $G$.   By Schafhauser's criterion,  the trivial crossed product $C^*$-algebra
 $\mathbf{C}\rtimes G\cong C^*(G)$ embeds into a unique  simple unital 
  Approximately  Finite-dimensional  (AF-)  $C^*$-algebra 
[Schafhauser 2020] \cite[Theorem C]{Sch1}, see also Lemma \ref{lm3.1} for an explicit construction.
  By an abuse of notation, we call the latter
  an (AF-) Etesi $C^*$-algebra and denote it by $\mathbb{E}_{\mathscr{M}}$. 
  \end{remark}
The aim of our note is a classification of the smooth structures on $\mathscr{M}$ 
based on  the K-theory of the Etesi $C^*$-algebra $\mathbb{E}_{\mathscr{M}}$. 
To formalize our results, recall that the AF $C^*$-algebra is 
an inductive limit of the multi-matrix 
$C^*$-algebras $M_{n_1}(\mathbf{C})\oplus\dots\oplus M_{n_k}(\mathbf{C})$
[Blackadar 1986] \cite[Section 7]{B}.  The AF-algebra is called 
stationary, if the inductive limit depends on  a single positive integer matrix 
$A\in GL(n, \mathbf{Z})$ [Blackadar 1986] \cite[Section 7.2]{B} or  \cite[Section 3.5.2]{N}.  
Our main result can be formulated as follows.  
\begin{theorem}\label{thm1.3}
The 
$C^*$-algebra $\mathbb{E}_{\mathscr{M}}$ is a stationary AF-algebra. 
\end{theorem}

\medskip
Let $\lambda_A>1$ be the Perron-Frobenius eigenvalue   of  the positive matrix $A$ 
defined by  $\mathbb{E}_{\mathscr{M}}$.
Consider a number field  $K=\mathbf{Q}(\lambda_A)$.   
The eigenvector $(v_1,\dots, v_n)$ corresponding to $\lambda_A$ can always be  scaled 
so that $v_i\in K$.  By $\mathfrak{m}:=\mathbf{Z}v_1+\dots+\mathbf{Z}v_n$ we understand 
 a $\mathbf{Z}$-module in the field  $K$ and 
 by  $\Lambda$  the ring of  endomorphisms of $\mathfrak{m}$. 
Let  $[\mathfrak{m}]$ be an ideal class of $\mathfrak{m}$  in the ring $\Lambda$.  
The K-theory of stationary  AF-algebras   says  that the 
 triples $(\Lambda, [\mathfrak{m}], K)$ in a one-to-one  correspondence with the Morita equivalence classes of  
 $\mathbb{E}_{\mathscr{M}}$ [Handelman 1981] \cite{Han2}, see also  \cite[Theorem 3.5.4]{N}. 
Let  $Br(K)$ be the Brauer group of the number field $K$,  i.e. a torsion abelian group 
of  the Morita equivalence  classes of the central simple algebras over $K$. 
By $S(k)$ we understand  the  connected sum of $k$ copies of $S^2\times S^2$. 
We assign an index $k\ge 0$ to each smoothing $\mathscr{M}_k$ of a topological 
4-manifold $\mathscr{M}_{top}$ using  Gompf's Stable Diffeomorphism Theorem
$\mathscr{M}_k\# S(k)\rightarrow \mathscr{M}_0\# S(k)$ [Gompf 1984] \cite[Theorem 1]{Gom1},
 where $\mathscr{M}_0$ is an arbitrary but fixed  smoothing of $\mathscr{M}_{top}$.
The sum  $\mathscr{M}_{k_1}\oplus\mathscr{M}_{k_2}:=\mathscr{M}_{k_1+k_2}$ (Definition \ref{dfn2.4}) 
gives the set of all smoothings of  $\mathscr{M}_{top}$ the stucture of an abelian monoid
(Corollary \ref{cor2.7}).
The choice of  $\mathscr{M}_0$ corresponds to such of an order unit $u\in K_0^+(\mathbb{E}_{\mathscr{M}})$
and therefore the scale $\Sigma(\mathbb{E}_{\mathscr{M}})$; we refer the reader to Section 2.2 for the definitions and details. 
The set $\Sigma(\mathbb{E}_{\mathscr{M}})$ itself is a torsion abelian group with the addition $\oplus$ 
and the neutral element  $\mathscr{M}_0$ (Corollary \ref{cor3.9}).
Theorem \ref{thm1.3} implies the following results.
\begin{corollary}\label{cor1.4}
Let $\mathscr{M}$ be a smooth 4-manifold, such that the $\mathbb{E}_{\mathscr{M}}$
is an infinite-dimensional $C^*$-algebra.  The following is true:

\bigskip
(i) Handelman triple $(\Lambda, [\mathfrak{m}], K)$ is an invariant 
of the homeomorphisms  of $\mathscr{M}$;

\medskip
(ii) Elements of the Brauer group $Br(K)$ parametrize
smooth structures on $\mathscr{M}_{top}$. 
In particular, all  smoothings  of $\mathscr{M}$
 form  a torsion abelian group under the operation $\oplus$  with the neutral element  $\mathscr{M}_0$. 
\end{corollary}

\bigskip
\begin{remark}
It is well known that the standard and exotic spheres form an abelian group under the connected 
sum [Milnor 1959] \cite[Corollary 2]{Mil1}. Part (ii) of Corollary \ref{cor1.4} extends Milnor's result to all compact 
smoothable 4-manifolds. Where this abelian group structure comes from? Informally, its existence  follows 
from Gompf's Theorem and Corollary \ref{cor3.6}.  Indeed, the $K$-theory says that the Morita equvalence classes of the 
$AF$-algebras carry the structure of an ordered abelian group (Section 2.2) which translates to a torsion group  
by the Minkowski question-mark function (Remark \ref{rmk2.1}).  Roughly speaking, Gompf's Theorem 
provides a \textit{realization} of the above algebraic addition  by a geometric operation of taking  connected sums of the compact smooth 4-manifolds 
(Corollary \ref{cor3.6}).  
\end{remark}
\begin{remark}
The Brauer group $Br(K)$ is known to classify the division algebras over $K$. 
In other words,  corollary \ref{cor1.4} (ii) defines a functor from the smooth 4-manifolds 
to the division algebras. Such a functor was  constructed independently 
using the Galois theory for non-commutative fields \cite{Nik1}. 
\end{remark}
The article is organized as follows.  Some preliminary facts can be found in Section 2. 
Theorem \ref{thm1.3} and corollary \ref{cor1.4}  are proved in Section 3. 
We conclude by remarks in Section 4. 

\section{Preliminaries}
In this section we briefly review  the $C^*$-algebras, their K-theory and the 4-dimensional manifolds.
We refer the reader to  [Blackadar 1986] \cite{B},  [Dixmier 1977]  \cite{D} and [Gompf 1984] \cite{Gom1} for the details.

\subsection{$C^*$-algebras}
The $C^*$-algebra is an algebra  $\mathscr{A}$ over $\mathbf{C}$ with a norm 
$a\mapsto ||a||$ and an involution $\{a\mapsto a^* ~|~ a\in \mathscr{A}\}$  such that $\mathscr{A}$ is
complete with  respect to the norm, and such that $||ab||\le ||a||~||b||$ and $||a^*a||=||a||^2$ for every  $a,b\in \mathscr{A}$.  
Each commutative $C^*$-algebra is  isomorphic
to the algebra $C_0(X)$ of continuous complex-valued
functions on some locally compact Hausdorff space $X$. 
Any other  algebra $\mathscr{A}$ can be thought of as  a noncommutative  
topological space. 

An {\it AF-algebra}  (Approximately Finite-dimensional $C^*$-algebra) is defined to
be the  norm closure of an ascending sequence of   finite dimensional
$C^*$-algebras $M_n$,  where  $M_n$ is the $C^*$-algebra of the $n\times n$ matrices
with entries in $\mathbf{C}$. Here the index $n=(n_1,\dots,n_k)$ represents
the  semi-simple matrix algebra $M_n=M_{n_1}\oplus\dots\oplus M_{n_k}$.
The ascending sequence mentioned above  can be written as 
\begin{equation}\label{eq2.1}
M_1\buildrel\rm\varphi_1\over\longrightarrow M_2
   \buildrel\rm\varphi_2\over\longrightarrow\dots,
\end{equation}
where $M_i$ are the finite dimensional $C^*$-algebras and
$\varphi_i$ the homomorphisms between such algebras.  
If $\varphi_i=Const$, then the AF-algebra $\mathscr{A}$ is called 
{\it stationary}. 

The homomorphisms $\varphi_i$ can be arranged into  a graph as follows. 
Let  $M_i=M_{i_1}\oplus\dots\oplus M_{i_k}$ and 
$M_{i'}=M_{i_1'}\oplus\dots\oplus M_{i_k'}$ be 
the semi-simple $C^*$-algebras and $\varphi_i: M_i\to M_{i'}$ the  homomorphism. 
One has  two sets of vertices $V_{i_1},\dots, V_{i_k}$ and $V_{i_1'},\dots, V_{i_k'}$
joined by  $a_{rs}$ edges  whenever the summand $M_{i_r}$ is contained in  $a_{rs}$
copies of the summand $M_{i_s'}$ under the embedding $\varphi_i$. 
As $i$ varies, one obtains an infinite graph called the   Bratteli diagram of the
AF-algebra.  The matrix $A=(a_{rs})$ is known as  a  partial multiplicity matrix;
an infinite sequence of $A_i$ defines a unique AF-algebra.
If   $\mathscr{A}$ is a stationary AF-algebra, then   $A_i=Const$
for all $i\ge 1$.

\subsection{K-theory of AF-algebras}
By $M_{\infty}(\mathscr{A})$ 
one understands the algebraic direct limit of the $C^*$-algebras 
$M_n(\mathscr{A})$ under the embeddings $a\mapsto ~\mathbf{diag} (a,0)$. 
The direct limit $M_{\infty}(\mathscr{A})$  can be thought of as the $C^*$-algebra 
of infinite-dimensional matrices whose entries are all zero except for a finite number of the
non-zero entries taken from the $C^*$-algebra $\mathscr{A}$.
Two projections $p,q\in M_{\infty}(\mathscr{A})$ are equivalent, if there exists 
an element $v\in M_{\infty}(\mathscr{A})$,  such that $p=v^*v$ and $q=vv^*$. 
The equivalence class of projection $p$ is denoted by $[p]$.   
We write $V(\mathscr{A})$ to denote all equivalence classes of 
projections in the $C^*$-algebra $M_{\infty}(\mathscr{A})$, i.e.
$V(\mathscr{A}):=\{[p] ~:~ p=p^*=p^2\in M_{\infty}(\mathscr{A})\}$. 
The set $V(\mathscr{A})$ has the natural structure of an abelian 
semi-group with the addition operation defined by the formula 
$[p]+[q]:=\mathbf{diag}(p,q)=[p'\oplus q']$, where $p'\sim p, ~q'\sim q$ 
and $p'\perp q'$.  The identity of the semi-group $V(\mathscr{A})$ 
is given by $[0]$, where $0$ is the zero projection. 
By the $K_0$-group $K_0(\mathscr{A})$ of the unital $C^*$-algebra $\mathscr{A}$
one understands the Grothendieck group of the abelian semi-group
$V(\mathscr{A})$, i.e. a completion of $V(\mathscr{A})$ by the formal elements
$[p]-[q]$.  The image of $V(\mathscr{A})$ in  $K_0(\mathscr{A})$ 
is a positive cone $K_0^+(\mathscr{A})$ defining  the order structure $\le$  on the  
abelian group  $K_0(\mathscr{A})$. The pair   $\left(K_0(\mathscr{A}),  K_0^+(\mathscr{A})\right)$
is known as a dimension group of the $C^*$-algebra $\mathscr{A}$. 
The scale $\Sigma(\mathscr{A})$ is the image in $K_0^+(\mathscr{A})$
of the equivalence classes of projections in the $C^*$-algebra $\mathscr{A}$. 
The $\Sigma(\mathscr{A})$ is a generating, hereditary and directed subset 
of  $K_0^+(\mathscr{A})$, i.e. (i) for each $a\in K_0^+(\mathscr{A})$ 
there exist $a_1,\dots, a_r\in\Sigma(\mathscr{A})$ such that 
$a=a_1+\dots+a_r$; (ii) if $0\le a\le b\in \Sigma(\mathscr{A})$, then $a\in\Sigma(\mathscr{A})$
and (iii) given $a,b\in\Sigma(\mathscr{A})$ there exists $c\in\Sigma(\mathscr{A})$,
such that $a,b\le c$.   Each  scale  can always be written as 
$\Sigma(\mathscr{A})=\{a\in K_0^+(\mathscr{A}) ~|~0\le a\le u\}$,
where $u$ is an  order unit of  $K_0^+(\mathscr{A})$.  
The pair  $\left(K_0(\mathscr{A}),  K_0^+(\mathscr{A})\right)$ and the
triple  $\left(K_0(\mathscr{A}),  K_0^+(\mathscr{A}), \Sigma(\mathscr{A})\right)$
are invariants of the Morita equivalence and isomorphism class of the 
$C^*$-algebra $\mathscr{A}$, respectively. 
If  $\mathbb{A}$ is an AF-algebra, then its scaled dimension group 
(dimension group, resp.) is a complete invariant of the isomorphism 
(Morita equivalence, resp.) class of $\mathbb{A}$, see e.g. \cite[Theorem 3.5.2]{N}.

Let $\tau$ be the canonical trace on the AF-algebra  $\mathbb{A}$. 
Such a trace induces a homomorphism $\tau_*: K_0(\mathbb{A})\to\mathbf{R}$
and  we let $\mathfrak{m}:=\tau_*( K_0(\mathbb{A}))\subset\mathbf{R}$.  
If    $\mathbb{A}$   is the stationary AF-algebra given by a matrix $A\in GL(n, \mathbf{Z})$, 
then $\mathfrak{m}$ is a $\mathbf{Z}$-module 
in the number field $K=\mathbf{Q}(\lambda_A)$ generated by the Perron-Frobenius 
eigenvalue $\lambda_A$  of  the matrix $A$.    The endomorphism ring of $\mathfrak{m}$
is denoted by $\Lambda$ and the ideal class of $\mathfrak{m}$ is denoted by $[\mathfrak{m}]$. 
The triple  $(\Lambda, [\mathfrak{m}], K)$
is an invariant of the Morita equivalence class of $\mathbb{A}$  [Handelman 1981] \cite{Han2}. 
\begin{remark}\label{rmk2.1}
Each  stationary AF-algebra defines  a torsion abelian group.
Indeed, let $?_n(x)$ be the $n$-dimensional Minkowski question-mark function, see
 [Minkowski 1904] \cite[p.172]{Min1} for $n=2$ and [Panti 2008] \cite[Theorem 3.5]{Pan1} for $n\ge 2$. 
The $?_n(x): [0,1]^{n-1}\to [0,1]^{n-1}$ is a continuous function with the following properties:
(i) $?_n(\mathbf{0})=\mathbf{0}$ and $?_n(\mathbf{1})=\mathbf{1}$, where $\mathbf{0}=(0,\dots,0)$
and  $\mathbf{1}=(1,\dots,1)$; (ii)  $?_n(\mathbf{Q}^{n-1})=(\mathbf{Z}[{1\over 2}])^{n-1}$ are dyadic rationals
and  (iii) $?_n(\mathcal{K}^{n-1})=(\mathbf{Q}-\mathbf{Z}[{1\over 2}])^{n-1}$, where $\mathcal{K}$ are algebraic 
numbers of degree $n$ over $\mathbf{Q}$. It is not hard to see, that (iv)   $?_n(\Delta)=\Delta$ is a monotone function,
where $\Delta=[0,1]$  is the normalized diagonal of the simplex $[0,1]^{n-1}$. 
Recall that $\tau_*(K_0(\mathbb{A}))=\mathfrak{m}$ and  $\tau_*(\Sigma(\mathbb{A}))=\mathfrak{m}\cap [0,1]$,
where $\tau$ is the canonical trace on the AF-algebra $\mathbb{A}$ and $\mathfrak{m}$ is a $\mathbf{Z}$-module
in the number field $K$. We assume that  $\tau_*(K_0(\mathbb{A}))\subset\Delta$. 
By the properties (iii) and (iv) of the Minkowski question-mark function, one gets the following
inclusion:
\begin{equation}\label{eq2.2}
\mathscr{Y}:=?_n(\tau_*(\Sigma(\mathbb{A})))\subset \mathbf{Q}/\mathbf{Z}. 
\end{equation}
\end{remark}
\begin{definition}\label{dfn2.2}
By the Minkowski group $Mi ~(K)$ of  stationary AF-algebra  we understand 
a torsion abelian group generated by the elements of  set  $\mathscr{Y}$. 
\end{definition}

\subsection{4-dimensional manifolds}
We denote by $\mathscr{M}$ a smooth 4-dimensional manifold and always 
assume $\mathscr{M}$ to be compact. 
Let $S^4$ be the 4-dimensional sphere and $X_g$ be a closed
2-dimensional orientable surface of genus  $g\ge 0$.
By the knotted surface  $\mathscr{X}:=X_{g_1}\cup\dots X_{g_n}$ in  $\mathscr{M}$  one understands  a 
transverse immersion of a collection of $n\ge 1$  surfaces  $X_{g_i}$ into $\mathscr{M}$. 
We  refer to $\mathscr{X}$ a  surface knot  if $n=1$ and a  surface link  if $n\ge 2$. 
\begin{theorem}\label{thm2.2}
{\bf ([Piergallini 1995] \cite{Pie1})} 
Each smooth 4-dimensional  manifold $\mathscr{M}$ is the 4-fold PL cover of the sphere $S^4$
 branched at the points of a knotted surface $\mathscr{X}\subset S^4$. 
\end{theorem}
Let $S^2$ be the 2-dimensional sphere.
By $S(k)$ we understand  a smooth  4-dimensional manifold corresponding to  a
 connected sum
\begin{equation}\label{eq2.3}
S(k):=\underbrace{(S^2\times S^2)\#\dots\# (S^2\times S^2)}_{\hbox{$k$ ~times}}.
\end{equation}
\begin{theorem}\label{thm2.3}
{\bf ([Gompf 1984] \cite{Gom1})}
Let $\mathscr{M}$ and $\mathscr{M}'$ be two different smoothings of a topological
4-manifold $\mathscr{M}_{top}$. Then for sufficiently large $k$ there exists an 
orientation-preserving diffeomorphism:
\begin{equation}\label{eq2.4}
\mathscr{M}\# S(k)\longrightarrow \mathscr{M}'\# S(k). 
\end{equation}
\end{theorem}
\begin{definition}\label{dfn2.4}
Let $\mathscr{M}_0$ be an arbitrary but fixed smoothing of the topological 4-manifold 
$\mathscr{M}_{top}$. We denote by  $\mathscr{M}_k$ a smoothing of 
$\mathscr{M}_{top}$, such that 
$\mathscr{M}_k\# S(k)\rightarrow \mathscr{M}_0\# S(k)$ 
is an orientation-preserving diffeomorphism given by formula (\ref{eq2.4}). 
Further, if for the pair of distinct smoothings  $\mathscr{M}_{k_1}$
and $\mathscr{M}_{k_2}$ there exists a new smoothing $\mathscr{M}_{k_1+k_2}$
of $\mathscr{M}_0$, we shall write: 
\begin{equation}\label{eq2.5}
\mathscr{M}_{k_1}\oplus\mathscr{M}_{k_2}=\mathscr{M}_{k_1+k_2}. 
\end{equation}
   \end{definition}
\begin{remark}
The set of smoothings of $\mathscr{M}_0$ satisfying (\ref{eq2.5}) is non-empty and countable unless 
$\mathscr{M}_{top}\cong S^4$, i.e. the case of (still open)  smooth Poincar\'e conjecture. 
The reader will see (Corollary \ref{cor3.6})  that $k$ is equal to the dimension of projection operators 
in the $C^*$-algebra $\mathbb{E}_{\mathscr{M}}$ and (\ref{eq2.5}) corresponds to  the orthogonal sum of such projectors. 
\end{remark}
\begin{corollary}\label{cor2.7}
The set of smoothings of $\mathscr{M}_{top}$ has the structure of an abelian monoid under the operation
(\ref{eq2.5})
with the neutral element $\mathscr{M}_0$. 
\end{corollary}
\begin{proof}
(i) Let us show that operation (\ref{eq2.5}) defines a semigroup. 
By definition \ref{dfn2.4},  the sum $\mathscr{M}_{k_1}\oplus\mathscr{M}_{k_2}$
is a smoothing of $\mathscr{M}_{top}$. In other words, the set of smoothings 
of $\mathscr{M}_{top}$ is closed under the operation   (\ref{eq2.5}), i.e. such a set
is a semigroup. 

\smallskip
(ii) The semigroup of item (i) is abelian, since $\mathscr{M}_{k_1}\oplus\mathscr{M}_{k_2}=
\mathscr{M}_{k_2}\oplus\mathscr{M}_{k_1}$. 

\smallskip
(iii) It follows from (\ref{eq2.5}), that $\mathscr{M}_k\oplus\mathscr{M}_0=\mathscr{M}_k$.
Thus $\mathscr{M}_0$ is the neutral element of the semigroup.
In other words,  our semigroup is a monoid.

\medskip
Corollary \ref{cor2.7} is proved.
 \end{proof}

\section{Proofs}
\subsection{Proof of theorem \ref{thm1.3}}
We shall split the proof in two  lemmas. 
\begin{lemma}\label{lm3.1}
The $C^*$-algebra $\mathbb{E}_{\mathscr{M}}$ is an AF-algebra. 
\end{lemma}
\begin{proof}
The  $\mathbb{E}_{\mathscr{M}}$ is an AF-algebra by definition, see Remark \ref{rmk1.2}.
The aim of Lemma \ref{lm3.1} is an explicit construction of  $\mathbb{E}_{\mathscr{M}}$
from the group $C^*$-algebra $C^*(G)$. 
Such a construction in terms of a von Neumann algebra
related to  $\mathbb{E}_{\mathscr{M}}$ is due to G\'abor Etesi. 
Namely, it  was shown that the von Neumann algebra 
is hyperfinite, see [Etesi 2017] \cite[Lemma 2.3]{Ete2}. 
Below  we adapt the proof 
to the case of the $C^*$-algebra  $\mathbb{E}_{\mathscr{M}}$.

\medskip
Let $\mathcal{G}:=\mathit{Diff}(\mathscr{M})/ \mathit{Diff}_0(\mathscr{M})$. 
Consider a profinite completion of the discrete group $\mathcal{G}$, i.e.

\begin{equation}\label{eq3.1}
\widehat{\mathcal{G}}:=\varprojlim \mathcal{G}/N,
\end{equation}
where $N$ ranges through the  open normal finite index subgroups of 
$\mathcal{G}$.  
Recall that if  $G$ is  a finite group,  then the  group algebra   $\mathbf{C}[G]$
has the form
\begin{equation}\label{eq3.2}
\mathbf{C}[G]\cong M_{n_1}(\mathbf{C})\oplus\dots\oplus M_{n_h}(\mathbf{C}),
\end{equation}
where $n_i$ are degrees of the irreducible representations of $G$ 
and $h$  is the total number of such representations [Serre 1967]
\cite[Proposition 10]{S}.   In view of (\ref{eq3.1}),  we have  
\begin{equation}\label{eq3.3}
\widehat{\mathcal{G}}\cong\varprojlim G_i,
\end{equation}
where $G_i$ is a finite group.  Consider a group algebra  
\begin{equation}\label{eq3.4}
\mathbf{C}[G_i]\cong M_{n_1}^{(i)}(\mathbf{C})\oplus\dots\oplus M_{n_h}^{(i)}(\mathbf{C})
\end{equation}
corresponding to $G_i$.  Notice that the $\mathbf{C}[G_i]$ is  a finite-dimensional 
$C^*$-algebra. The inverse limit (\ref{eq3.3})  defines an 
ascending sequence of the finite-dimensional $C^*$-algebras of the form
\begin{equation}\label{eq3.5}
\varprojlim M_{n_1}^{(i)}(\mathbf{C})\oplus\dots\oplus 
M_{n_h}^{(i)}(\mathbf{C}).  
\end{equation}

The group $C^*$-algebra $C^*(\widehat{\mathcal{G}})$ of the
profinite group $\widehat{\mathcal{G}}$ is the norm closure of the 
group algebra $\mathbf{C}[\widehat{\mathcal{G}}]$ 
[Dixmier 1977]  \cite[Section 13.9]{D}.  One concludes from (\ref{eq3.5}) 
that the  $C^*(\widehat{\mathcal{G}})$  embeds into an AF-algebra. 

On the other hand, each non-trivial canonical homomorphism 
 $\mathcal{G}\to \widehat{\mathcal{G}}$ gives rise to an extension of the
 $C^*$-algebras:
\begin{equation}\label{eq3.6}
C^*(\mathcal{G})\to C^*(\widehat{\mathcal{G}})\to \mathscr{B}. 
 \end{equation}
Since $C^*(\widehat{\mathcal{G}})$ embeds into an AF-algebra,  both $C^*(\mathcal{G})$
and $\mathscr{B}$ must have such a property [Handelman 1982] \cite[Lemma I.5(a)]{Han1}. 
It remains to recall the definition of $\mathbb{E}_{\mathscr{M}}$,
see Remark \ref{rmk1.2}. 
 Lemma \ref{lm3.1} is proved. 
\end{proof}
\begin{lemma}\label{lm3.2}
The $\mathbb{E}_{\mathscr{M}}$ is a stationary AF-algebra. 
\end{lemma}
\begin{proof}
Roughly speaking, this fact is a consequence of the Piergallini Theorem \cite{Pie1}
followed by the Handelman's Criterion \cite[Theorem II (iii)]{Han2}. 
Let us pass to the detailed argument.

\medskip
(i)  Let $S^4$ be the 4-dimensional sphere. 
By the Piergallini Theorem \ref{thm2.2}, there exists a 4-fold  covering map
$\mathscr{M}\to S^4$ branched at the points of a knotted surface
$\mathscr{X}$ defined  by an embedding:
\begin{equation}\label{eq3.7}
\mathscr{X}\hookrightarrow S^4. 
 \end{equation}
In view of the inclusion  $\mathit{Diff}(S^4)\subset \mathit{Diff}(S^4-\mathscr{X})$,
one gets from the map  $\mathscr{M}\to S^4$ an injective homomorphism of the $C^*$-algebras:
\begin{equation}\label{eq3.8}
 \mathbb{E}_{S^4}\hookrightarrow
\mathbb{E}_{\mathscr{M}}.  
 \end{equation}

\bigskip
(ii) Let us show   that $\mathbb{E}_{S^4}\cong \mathbf{C}$. Indeed, 
since $\mathit{Diff}(S^4)\cong \mathit{Diff}_0(S^4)$, 
the group $\mathit{Diff}(S^4)/ \mathit{Diff}_0(S^4)$
is trivial.  In particular,   the group $C^*$-algebra $\mathbb{E}_{S^4}$ is  commutative. 
The Gelfand Theorem says that $\mathbb{E}_{S^4}\cong C_0(X)$, where $C_0(X)$ 
is the $C^*$-algebra of continuous complex-valued functions on a locally compact 
Hausdorff space $X$, see Section 2.1.  But $X\cong\mathbf{pt}$ is a singleton and therefore
$C_0(\mathbf{pt})\cong \mathbf{C}$.  Thus one gets $\mathbb{E}_{S^4}\cong \mathbf{C}$. 
(Equivalently, the group $C^*$-algebra of trivial group is isomorphic to $\mathbf{C}$.)

\bigskip
(iii) The  AF-algebra  $\mathbb{E}_{S^4}$ is given by an ascending 
sequence (\ref{eq2.1}) of the form: 
\begin{equation}\label{eq3.9}
\mathbf{C}\buildrel\rm \mathbf{1}\over\longrightarrow \mathbf{C}
   \buildrel\rm \mathbf{1}\over\longrightarrow\dots,
\end{equation}
where $\mathbf{1}$ is the identity homomorphism. 
It follows from (\ref{eq3.9}) that the $\mathbb{E}_{S^4}$
is a stationary AF-algebra. 

\bigskip
(iv)  Since $K_0(\mathbf{C})\cong\mathbf{Z}$, we conclude that the dimension group 
of the AF-algebra $\mathbb{E}_{S^4}$ is isomorphic to 
$(\mathbf{Z}, \mathbf{Z}^+)$, where $\mathbf{Z}^+$ is the semi-group 
of positive integers. It is easy to see, that the Handelman triple corresponding to
 the stationary AF-algebra $\mathbb{E}_{S^4}$  has the form $(\mathbf{Z}, [\mathbf{Z}], \mathbf{Q})$.

 \bigskip
 (v)  On the other hand, the map  (\ref{eq3.8}) induces an inclusion of the abelian groups:
 \begin{equation}\label{eq3.10}
K_0(\mathbb{E}_{S^4})\subset  K_0(\mathbb{E}_{\mathscr{M}}). 
\end{equation}
Moreover, if $\tau$ is the canonical trace on the AF-algebra $\mathbb{E}_{\mathscr{M}}$,
one gets from (\ref{eq3.10}) the following  inclusion of the additive groups  of the real line:
 \begin{equation}\label{eq3.11}
\mathbf{Z}\subset  \tau_*(K_0(\mathbb{E}_{\mathscr{M}})). 
\end{equation}

\bigskip
(vi) Since $\mathbf{Z}$ is a ring, the group inclusion (\ref{eq3.11}) can 
be extended to such of the rings. But the only finite degree extension 
of the ring $\mathbf{Z}$ coincides (up to a scaling constant) 
with an order $\Lambda$ in the number field $K=\Lambda\otimes\mathbf{Q}$. 
We conclude  that 
 \begin{equation}\label{eq3.12}
\tau_*(K_0(\mathbb{E}_{\mathscr{M}}))\subset K, 
\end{equation}
where $K$ is a real number field. 

\bigskip
(vii) To finish the proof of lemma \ref{lm3.2}, it remains to apply the result of [Handelman 1981] \cite[Theorem II (iii)]{Han2}
saying  that  condition  (\ref{eq3.12}) is equivalent to the AF-algebra $\mathbb{E}_{\mathscr{M}}$ 
to be of a stationary type. 

\bigskip
Lemma \ref{lm3.2} is proved.
\end{proof}

\smallskip
\begin{remark}\label{rmk3.3}
The Etesi $C^*$-algebra $\mathbb{E}_{\mathscr{M}}$  is simple, i.e. has only trivial two-sided ideals. 
This fact follows from lemma \ref{lm3.2} and strict positivity of the matrix $A$
corresponding to the stationary AF-algebra.   
\end{remark}

\smallskip
Returning to the proof of theorem \ref{thm1.3},  we apply lemmas \ref{lm3.1} and \ref{lm3.2}.
Theorem \ref{thm1.3} follows.

\subsection{Proof of corollary \ref{cor1.4}}
We split the proof in a series of lemmas. 
\begin{lemma}\label{lm3.4}
The Etesi $C^*$-algebras satisfy an isomorphism:
 \begin{equation}\label{eq3.13}
\mathbb{E}_{\mathscr{M}_1\#\mathscr{M}_2}\cong \mathbb{E}_{\mathscr{M}_1}\otimes \mathbb{E}_{\mathscr{M}_2},
\end{equation}
where  $\#$   is the connected sum of manifolds 
and  $\otimes$ is the tensor product of $C^*$-algebras. 
 \end{lemma}
\begin{proof}
We let $\mathcal{G}:=\mathit{Diff}(\mathscr{M}_1\#\mathscr{M}_2)/ \mathit{Diff}_0(\mathscr{M}_1\#\mathscr{M}_2)$.
It is not hard to see, that  
 \begin{equation}\label{eq3.14}
\mathcal{G}=\mathcal{G}_1\times\mathcal{G}_2,
\end{equation}
where 
 $\mathcal{G}_1=\mathit{Diff}(\mathscr{M}_1)/ \mathit{Diff}_0(\mathscr{M}_1)$
 and $\mathcal{G}_2=\mathit{Diff}(\mathscr{M}_2)/ \mathit{Diff}_0(\mathscr{M}_2)$. 
It is well known that the group ring $\mathbf{C}[\mathcal{G}]$ of the product 
(\ref{eq3.14}) is given by the formula: 
 \begin{equation}\label{eq3.15}
\mathbf{C}[\mathcal{G}]\cong \mathbf{C}[\mathcal{G}_1]\otimes\mathbf{C}[\mathcal{G}_2]. 
\end{equation}
Since the $\mathbb{E}_{\mathscr{M}}$  is a nuclear $C^*$-algebra, the norm closure of 
a self-adjoint representation of (\ref{eq3.15}) defines an isomorphism
$\mathbb{E}_{\mathscr{M}_1\#\mathscr{M}_2}\cong \mathbb{E}_{\mathscr{M}_1}\otimes \mathbb{E}_{\mathscr{M}_2}$. 
Lemma \ref{lm3.4} is proved. 
\end{proof}

\begin{lemma}\label{lm3.5}
The Etesi $C^*$-algebra of the 4-manifold $S^2\times S^2$ is 
given by the formula:
 \begin{equation}\label{3.16}
\mathbb{E}_{S^2\times S^2}\cong M_4(\mathbf{C}).
\end{equation}
\end{lemma}
\begin{proof}
(i) It follows from Piergallini’s Theorem \ref{thm2.2},  that the 4-manifold $S^2\times S^2$ is a 4-fold cover 
of the 4-sphere $S^4$  as one can see by factoring this covering geometrically as:
 \begin{equation}\label{eq3.17}
S^2 \times S^2 \cong \mathbf{C}P^1 \times \mathbf{C}P^1 \to \mathbf{C}P^2 \to S^4.
\end{equation}
The covering map (\ref{eq3.17}) induces a homomorphism of the $C^*$-algebras
$\mathbb{E}_{S^2\times S^2}\to \mathbb{E}_{S^4}$ and a homomorphism of 
the corresponding abelian groups:
 \begin{equation}\label{eq3.18}
K_0(\mathbb{E}_{S^2\times S^2})\to K_0(\mathbb{E}_{S^4}).
\end{equation}
From (\ref{eq3.9}) one gets  $K_0(\mathbb{E}_{S^4})\cong\mathbf{Z}$. In view of (\ref{eq3.17}),  
the kernel of the map (\ref{eq3.18}) is isomorphic to $\mathbf{Z}/4\mathbf{Z}$. 
Thus one gets an isomorphism:
 \begin{equation}\label{eq3.19}
K_0(\mathbb{E}_{S^2\times S^2})\cong\mathbf{Z}.
\end{equation}

\medskip
(ii) By theorem \ref{thm1.3},   the $\mathbb{E}_{S^2\times S^2}$ is a stationary
AF-algebra.  If $A$ is the corresponding matrix, then by (\ref{eq3.19}) 
the eigenvalues of $A$ must be rational and equal to each other.
In other words,   the  AF-algebra   $\mathbb{E}_{S^2\times S^2}$
 corresponds  to an ascending  sequence (\ref{eq2.1}) of the form: 
\begin{equation}\label{eq3.20}
M_2(\mathbf{C})\otimes M_2(\mathbf{C})\buildrel\rm 
\mathbf{1}
\over\longrightarrow M_2(\mathbf{C})\otimes M_2(\mathbf{C})
   \buildrel\rm 
   \mathbf{1},
   \over\longrightarrow\dots
\end{equation}
where $\mathbf{1}=\left(\begin{smallmatrix} 1 & 0\cr 0 & 1\end{smallmatrix}\right)\otimes
\left(\begin{smallmatrix} 1 & 0\cr 0 & 1\end{smallmatrix}\right)$. 
Clearly, the inductive limit (\ref{eq3.20}) corresponds to the $C^*$-algebra $M_2(\mathbf{C})\otimes M_2(\mathbf{C})$.
We conclude that   $\mathbb{E}_{S^2\times S^2}\cong M_2(\mathbf{C})\otimes M_2(\mathbf{C})\cong  M_4(\mathbf{C})$. 
Lemma \ref{lm3.5} is proved. 
\end{proof}

\begin{corollary}\label{cor3.6}
The Etesi $C^*$-algebra $\mathbb{E}_{\mathscr{M}\# S(k)}$ is Morita equivalent to $\mathbb{E}_{\mathscr{M}}$. 
\end{corollary}
\begin{proof}
Recall that the 4-manifold $S(k)=(S^2\times S^2)\#\dots\# (S^2\times S^2)$
is a connected sum of the $k$ copies of $S^2\times S^2$. 
From lemma \ref{lm3.5} and formula (\ref{eq3.13}) one gets an isomorphism:
 \begin{equation}\label{eq3.21}
 \mathbb{E}_{S(k)}\cong \underbrace{M_4(\mathbf{C})\otimes\dots\otimes M_4(\mathbf{C})}_{\hbox{$k$ ~times}}
 \cong M_{4^k}(\mathbf{C}). 
 \end{equation}
 
 If $\mathscr{M}$ is a smooth 4-manifold, then by lemma \ref{lm3.4} and formula (\ref{eq3.21}) 
 one obtains  an isomorphism: 
 \begin{equation}\label{eq3.22}
 \mathbb{E}_{\mathscr{M}\# S(k)}
\cong \mathbb{E}_{\mathscr{M}}\otimes M_{4^k}(\mathbf{C}). 
 \end{equation}

In other words, the  $C^*$-algebras $\mathbb{E}_{\mathscr{M}}$ and $\mathbb{E}_{\mathscr{M}\# S(k)}$
are Morita equivalent.
Corollary \ref{cor3.6} is proved. 
\end{proof}

\begin{lemma}\label{lm3.7}
The triple $(\Lambda, [\mathfrak{m}], K)$
is a topological invariant of  manifold $\mathscr{M}$.  
\end{lemma}
\begin{proof}
Our proof is based on 
 Gompf's Stable Diffeomorphism Theorem \ref{thm2.3} 
and corollary \ref{cor3.6}.   We shall proceed in two steps.
 
 \medskip
 (i) Recall that by Theorem \ref{thm2.3}  for any two  smoothings 
 $\mathscr{M}$ and  $\mathscr{M}'$ of a topological manifold $\mathscr{M}_{top}$
 one can find an integer  $k\ge 0$,   such that $\mathscr{M}\# S(k)$ and $\mathscr{M}'\# S(k)$
 are diffeomorphic. The corresponding $C^*$-algebras are  isomorphic: 
 \begin{equation}\label{eq3.23}
 \mathbb{E}_{\mathscr{M}\# S(k)}
\cong \mathbb{E}_{\mathscr{M}'\# S(k)}. 
 \end{equation}

\medskip
In view of corollary \ref{cor3.6} and formula (\ref{eq3.23}), the $C^*$-algebra   $\mathbb{E}_{\mathscr{M}}$ is 
Morita equivalent to  $\mathbb{E}_{\mathscr{M}\# S(k)}$ and the $C^*$-algebra 
$\mathbb{E}_{\mathscr{M}\# S(k)}$ is Morita equivalent to $\mathbb{E}_{\mathscr{M}'}$. 
Therefore the $C^*$-algebras $\mathbb{E}_{\mathscr{M}}$ and $\mathbb{E}_{\mathscr{M}'}$
are Morita equivalent  by the transitivity property.

 Thus the Morita equivalence class of the Etesi $C^*$-algebra 
$\mathbb{E}_{\mathscr{M}}$ consists of  all 4-dimensional manifolds which are homeomorphic but 
not diffeomorphic to each other. 

\medskip
(ii)  Recall that the $\mathbb{E}_{\mathscr{M}}$ is a stationary AF-algebra (lemma \ref{lm3.2}). 
The Morita equivalence classes of such AF-algebras are described by the Handelman triples $(\Lambda, [\mathfrak{m}], K)$,
see [Handelman 1981] \cite{Han2},  \cite[Theorem 3.5.4]{N} or Section 2.2. Comparing this fact with the result 
of item (i), we conclude that the $(\Lambda, [\mathfrak{m}], K)$ is an invariant of the topological type 
of the manifold $\mathscr{M}$. 

\medskip
Lemma \ref{lm3.7} is proved. 
\end{proof}
\begin{remark}\label{rmk3.8}
Lemma \ref{lm3.7} says that the topological type $\mathscr{M}_{top}$
of manifold $\mathscr{M}$ depends on the Morita equivalence class of the $C^*$-algebra
 $\mathbb{E}_{\mathscr{M}}$.  Likewise,  
  distinct smoothings  of  $\mathscr{M}_{top}$  are indexed  by  the isomorphism 
 classes  inside  given Morita equivalence class
 of $\mathbb{E}_{\mathscr{M}}$. 
\end{remark}

\begin{lemma}\label{lm3.8}
Let $K$ be a number field and let $Mi~(K)$ be the Minkowski group, see definition \ref{dfn2.2}. 
Then: 

\bigskip
(i) the map $K\to Mi~(K)$ is a covariant functor which maps isomorphic number 
fields $K$ to the isomorphic torsion abelian groups $Mi~(K)$; 

\medskip
(ii) $Mi ~(K)\cong Br(K)$.  

\end{lemma}
\begin{proof}
(i) Let $K$ be a number field corresponding to the Handelman triple  $(\Lambda, [\mathfrak{m}], K)$. 
Denote by $[\mathbb{A}]$ the Morita equivalence class of  stationary AF-algebras defined by 
the triple $(\Lambda, [\mathfrak{m}], K)$. 
Since $\tau$ is the canonical trace (i.e. $\tau_*(u)=1$ for the order unit $u\in K_0^+(\mathbb{A})$),
we conclude that $\tau_*(\Sigma (\mathbb{A}))\subset [0,1]$ does not depend on  $\mathbb{A}\in [\mathbb{A}]$. 
Thus from (\ref{eq2.2}) one gets a correctly defined map  $K\to \mathscr{Y}:=Mi~(K)$. 
It can be verified directly that if $K'$ is a real embedding of $K$, then $Mi~(K')\cong Mi~(K)$. 
Item (i) is proved. 

\bigskip
(ii) Let $Br(K)$ be the Brauer group of a number field $K$. 
It is well known,  that the map $K\to Br(K)$ is a covariant functor which maps isomorphic number 
fields $K$ to the isomorphic torsion abelian groups $Br(K)$. Comparing with item (i) we
conclude that there exists a natural transformation between these two functors, see e.g. \cite[Section 2.3]{N}. 
In particular, such a transformation implies an isomorphism of the abelian groups $Br(K)$ and $Mi~(K)$. 
Item (ii) and lemma \ref{lm3.8} are proved.   
\end{proof}
\begin{remark}
Lemma \ref{lm3.8} can be viewed as part of a  correspondence between the algebraic K-theory 
(e.g. the Milnor-Voevodsky K-theory) and  the Galois cohomology. 
This subject is outside the scope of present note. 
\end{remark}
\begin{corollary}\label{cor3.9}
Distinct smoothings   of $\mathscr{M}$ are classified  by  
the elements of the Brauer group $Br(K)$.  
In particular,   all  smoothings of  $\mathscr{M}$ 
form a torsion abelian group with the summation operation defined by formula (\ref{eq2.5})
and the neutral element $\mathscr{M}_0$,
see Section 2.3.    
\end{corollary}
\begin{proof}
(i) Recall that  from (\ref{eq2.2}) we have:
 \begin{equation}\label{eq3.24}
 \mathscr{Y}=?_n(\tau_*(\Sigma(\mathbb{E}_{\mathscr{M}}))).
\end{equation}
Since the Minkowski function $?_n(x)$ is monotone, formula (\ref{eq3.24}) 
defines a bijective correspondence between  generators of the torsion abelian groups 
$Mi~(K)\cong Br~(K)$ and the scale  $\Sigma(\mathbb{E}_{\mathscr{M}})$ of the Etesi
$C^*$-algebra $\mathbb{E}_{\mathscr{M}}$.

\bigskip
(ii)  Recall that  the scale $\Sigma(\mathbb{E}_{\mathscr{M}})$ is a generating
subset  of   $K_0^+(\mathbb{E}_{\mathscr{M}})$, 
i.e. for each $a\in K_0^+(\mathbb{E}_{\mathscr{M}})$ 
there exist $a_1,\dots, a_r\in \Sigma(\mathbb{E}_{\mathscr{M}})$ such that 
$a=a_1+\dots+a_r$.  We extend the correspondence of item  (i) to 
a bijective map between the elements of the Brauer group
$Br(K)$ and the elements of positive cone $K_0^+(\mathbb{E}_{\mathscr{M}})$. 

\bigskip
(iii)  It is known that the pair $(K_0(\mathbb{E}_{\mathscr{M}}), K_0^+(\mathbb{E}_{\mathscr{M}}))$ 
is invariant of the Morita equivalence class of the AF-algebra  $\mathbb{E}_{\mathscr{M}}$,
while the triple $(K_0(\mathbb{E}_{\mathscr{M}}), K_0^+(\mathbb{E}_{\mathscr{M}}), \Sigma(\mathbb{E}_{\mathscr{M}}))$
is invariant of the isomorphism class of   $\mathbb{E}_{\mathscr{M}}$, see Section 2.2. 
Moreover, the scale  can be written as  $\Sigma(\mathbb{E}_{\mathscr{M}})=\{a\in K_0^+(\mathbb{E}_{\mathscr{M}}) ~|~0\le a\le u\}$,
where   $u\in K_0^+(\mathbb{E}_{\mathscr{M}})$ is fixed. Thus running through all $u\in K_0^+(\mathbb{E}_{\mathscr{M}})$
one gets all possible scales $\Sigma(\mathbb{E}_{\mathscr{M}})$ and vice versa. In other words, 
the elements  $u\in K_0^+(\mathbb{E}_{\mathscr{M}})$ parametrize isomorphism classes of 
$\mathbb{E}_{\mathscr{M}}$ within its Morita equivalence class.

\bigskip
(iv) To finish the proof of corollary \ref{cor3.9}, it remains to recall remark \ref{rmk3.8}.  
Indeed, combing \ref{rmk3.8} with item (iii) we conclude that 
different smooth structures on  $\mathscr{M}$ are in bijection with
the elements of $K_0^+(\mathbb{E}_{\mathscr{M}})$.  
Moreover,  the $K_0^+(\mathbb{E}_{\mathscr{M}})$
has  structure of a torsion abelian group  isomorphic to the Brauer 
group $Br(K)$, see item (ii).  Corollary \ref{cor3.9} is proved. 
\end{proof}

\bigskip
Corollary \ref{cor1.4} follows from lemma \ref{lm3.7} and corollary \ref{cor3.9}.

\section{Remarks}
\begin{remark}\label{rmk4.1}
The $\mathbb{E}_{\mathscr{M}}$  can be defined via  unitary representation of 
the group  $\mathit{Diff}(\mathscr{M})$ by the bounded linear operators on a Hilbert space 
[Etesi 2016] \cite[Section 2]{Ete1}.  Since $\mathit{Diff}(\mathscr{M})$    is a Fr\'echet manifold,
the standard construction of  the group $C^*$-algebra  fails in general 
 [Blackadar 1986] \cite[Section 11.1]{B};  hence  definition \ref{dfn1.1}. 
However, the two definitions are equivalent from the standpoint 
of representation theory. We refer the reader to [Etesi 2021] \cite{Ete2}
for a  remarkable new topological invariant of the 4-manifolds  given by 
 the Murray - von Neumann coupling constant of  $\mathbb{E}_{\mathscr{M}}$.  
\end{remark}
\begin{remark}\label{rmk4.2}
Let $\mathscr{M}$ be the 4-dimensional sphere $S^4$. 
Since  $\mathbb{E}_{S^4}\cong\mathbf{C}$ is a finite-dimensional $C^*$-algebra,  
corollary \ref{cor1.4} fails.  On the other hand, it is  known that  $K_0(\mathbb{E}_{S^4})\cong\mathbf{Z}$ 
and $K\cong\mathbf{Q}$.  Denote by  $\mathscr{S}$  a subgroup of the Brauer group $Br(\mathbf{Q})$
consisting of  all  smoothings  of  $S^4$. 
Assuming an analog of  \ref{cor1.4},  one can recast 
 the  smooth Poincar\'e Conjecture as follows: 
Is the group $\mathscr{S}$  trivial? 
\end{remark}

\section*{Data availability}
  
  Data sharing not applicable to this article as no datasets were generated or analyzed during the current study.
   
\section*{Conflict of interest}
On behalf of all co-authors, the corresponding author states that there is no conflict of interest.
  

\section*{Funding declaration}
The author was partly supported by the NSF-CBMS grant 2430454.


\bibliographystyle{amsplain}


\end{document}